\documentclass{amsart}
\usepackage{amssymb}
\usepackage{amsmath}
\usepackage{graphicx}
\usepackage{amsfonts}
\usepackage{color}
\usepackage{graphicx}

\providecommand{\U}[1]{\protect\rule{.1in}{.1in}}
\newtheorem{theorem}{Theorem}
\theoremstyle{plain}
\newtheorem{algorithm}{Algorithm}

\newtheorem{example}{Example}
\newtheorem{lemma}{Lemma}
\newtheorem{proposition}{Proposition}
\numberwithin{equation}{section}

\input{tcilatex}

\begin{document}
\title{Constructing the fixed point of a non-isometric plane similarity }
\author{Zhigang Han}
\address{Department of Mathematics\\
Millersville University of Pennsylvania\\
Millersville, PA. 17551}
\email{zhigang.han@millersville.edu}
\author{Ronald Umble}
\address{Department of Mathematics\\
Millersville University of Pennsylvania\\
Millersville, PA. 17551}
\email{ron.umble@millersville.edu}
\date{September 16, 2014}

\begin{abstract}
We present an algorithm for constructing the fixed point of a general
non-isometric similarity of the plane.
\end{abstract}

\maketitle

\section{Introduction}

A majority of states across the USA are in the process of implementing the
\emph{Common Core State Standards for Mathematics} \cite{CCSSI}. An
essential component of the geometry standard is the study of plane
similarities. To quote the Standards, \textquotedblleft \emph{Given two
similar two-dimensional figures, describe a sequence that exhibits the
similarity between them}.\textquotedblright\ [CCSS.MATH.CONTENT.8.G.A.4] If
the given figures are congruent, the similarity $\alpha $ relating them is a
distance-preserving transformation, called an \emph{isometry, }and by the
Classification Theorem for Plane Isometries, $\alpha $ is a translation, a
rotation, a reflection, or a glide reflection \cite{Martin}. On the other
hand, if the given figures are not congruent, $\alpha $ is a non-isometric
similarity, and by the Classification Theorem for Plane Similarities, $%
\alpha $ is a stretch, a stretch rotation, or a stretch reflection \cite%
{Martin}. Nevertheless, in either case, it is easy to determine the \emph{%
similarity type} of $\alpha $ at a glance. For a discussion of the \emph{%
Similarity Recognition Problem}, see \cite{Umble-Han}.

The precise similarity relating two similar non-congruent plane figures can
be difficult to determine, however. While the fact that \emph{every
non-isometric similarity has a fixed point} is an important clue (see
Theorem \ref{thm4.7} below), the proof only establishes existence, and it is
not immediately clear how to go about constructing the fixed point. And this
is essential -- once the fixed point is determined, the additional
requirements of the definition follow immediately. In this article we
present a straight forward algorithm for constructing the fixed point of a
non-isometric similarity.\medskip

\section{Notation and Definitions}

Let $r\in \mathbb{R}$ and let $P$ and $Q$ be arbitrarily chosen points. A
(plane) \textbf{similarity of ratio} $r>0$ is a transformation $\alpha :%
\mathbb{R}^{2}\rightarrow \mathbb{R}^{2}$ such that if $P^{\prime }=\alpha
(P)$ and $Q^{\prime }=\alpha (Q),$ then $P^{\prime }Q^{\prime }=rPQ.$ An
\textbf{isometry }is a similarity of ratio $1.$ One can show that
similarities are bijective transformations that preserve betweeness; for a
proof see \cite{Umble-Han}.

Let $C$ be a point. The \textbf{stretch} \textbf{about} $C$\textbf{\ of ratio%
} $r>0$ is the transformation $\xi _{C,r}:\mathbb{R}^{2}\rightarrow \mathbb{R%
}^{2}$ with the following properties:\ (1) $\xi _{C,r}\left( C\right) =C;$
(2) \emph{if }$P\neq C,$ \emph{then }$P^{\prime }=\xi _{C,r}\left( P\right)
\ $\emph{is the unique point on}\textit{\ }$\overrightarrow{CP}$\textit{\ }%
\emph{such that}\textit{\ }$CP^{\prime }=rCP.$ The point $C$ is called the
\textbf{center }of the stretch \textup{(see Figure 1), and i}t is easy to
check that $\xi _{C,r}$ is a similarity of ratio $r$.

\medskip

\begin{center}
\includegraphics[
height=1.4996in,
width=2.8911in
]
{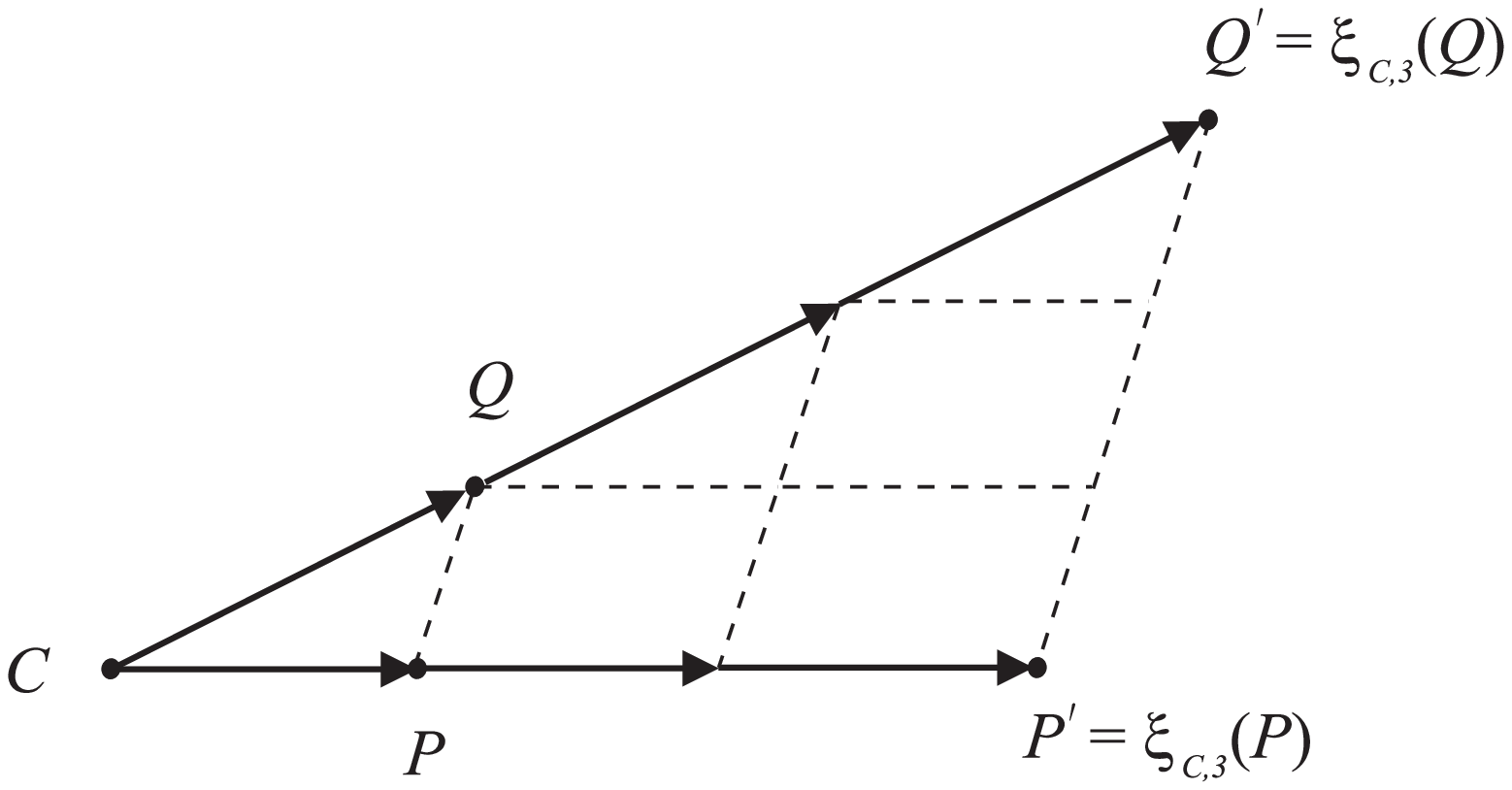}
\\\bigskip
\upshape{Figure 1. The stretch about $C$
of ratio $3$.}
\end{center}
\medskip

Let $\theta \in \mathbb{R}$. The \textbf{rotation} \textbf{about} $C$\textbf{%
\ through angle }$\theta $ is the transformation $\rho _{C,\theta }:\mathbb{R%
}^{2}\rightarrow \mathbb{R}^{2}$ with the following properties: (1) $\rho
_{C,\theta }\left( C\right) =C;$ (2)\ \emph{if} $P\neq C$ \textit{and} $%
P^{\prime }=\rho _{C,\theta }\left( P\right) ,$ \emph{then }$CP^{\prime }=CP$
\emph{and} \emph{the angle from} $\overrightarrow{CP}$ \emph{to} $%
\overrightarrow{CP^{\prime }}$ \emph{is }$\theta $. The point $C$ is called
the \textbf{center }of rotation \textup{(see Figure 2)}.

\medskip

\begin{center}
\includegraphics[
height=1.2237in,
width=1.6233in
]
{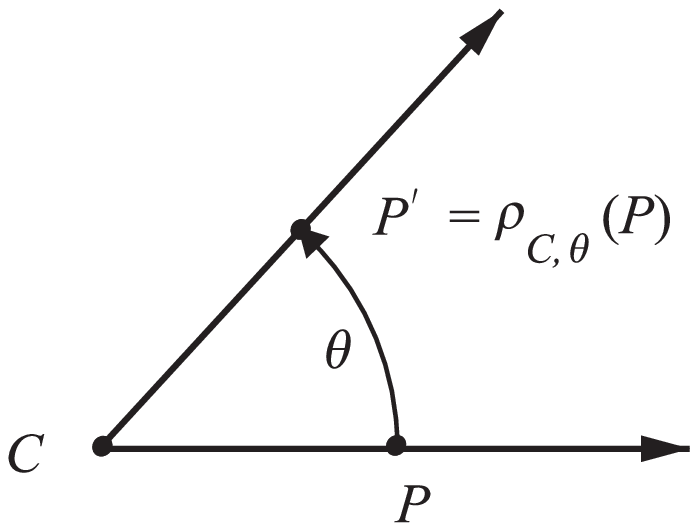}
\\\bigskip
\upshape{Figure 2. The rotation about $C$
through angle $\protect\theta .$}
\end{center}
\medskip

The \textbf{halfturn about }$C$ is the transformation $\varphi _{C}:\mathbb{R%
}^{2}\rightarrow \mathbb{R}^{2}$ defined by $\varphi _{C}\left( P\right)
=\rho _{C,180^{\circ }}\left( P\right) .$ A \textbf{dilation }about $C$ of
ratio $r>0$ is a transformation $\delta _{C,r}:\mathbb{R}^{2}\rightarrow
\mathbb{R}^{2}$ that is either a stretch or a stretch followed by a
halfturn, i.e., $\delta _{C,r}=\xi _{C,r}$ or $\delta _{C,r}=\varphi
_{C}\circ \xi _{C,r}.$ The point $C$ is the \textbf{center }of dilation, and
a dilation of ratio $1$ is either the identity $\iota $ or a halfturn.
Furthermore, $\delta _{C,r}$ is a similarity of ratio $r$ since the
composition of a similarity of ratio $r$ with a similarity of ratio $1$ is a
similarity of ratio $r.$ When $C$ is \textbf{between} points $P$ and $Q$ we
write $P-C-Q.$ Note that if $C-P-Q,$ there is a real number $r>0$ such that $%
Q=\xi _{C,r}\left( P\right) ,$ and if $P-C-Q,$ there is a real number $r>0$
such that $Q=\left( \varphi _{C}\circ \xi _{C,r}\right) \left( P\right) .$
Thus in either case, $Q=\delta _{C,r}\left( P\right) .$

Let $m$ be a line. The \textbf{reflection in line }$m$ is the transformation
$\sigma _{m}:\mathbb{R}^{2}\rightarrow \mathbb{R}^{2}$ with the following
properties:\ (1) \emph{If }$P$\emph{\ is on }$m,$\emph{\ then }$\sigma
_{m}\left( P\right) =P;$\emph{\ (2) if }$P$\emph{\ is off }$m$\emph{\ and }$%
P^{\prime }=\sigma _{m}\left( P\right) ,$\emph{\ then }$m$\emph{\ is the
perpendicular bisector of} $\overline{PP^{\prime }}.$ The line $m$ is the
\textbf{axis} of reflection (see Figure 3). A \textbf{stretch rotation} is a
non-trivial stretch about a point $C$ followed by a non-trivial rotation
about $C.$ A \textbf{stretch reflection} is a non-trivial stretch about a
point $C$ followed by a reflection in some line through $C.$

\medskip

\begin{center}
\includegraphics[
height=1.4529in,
width=2.3877in
]
{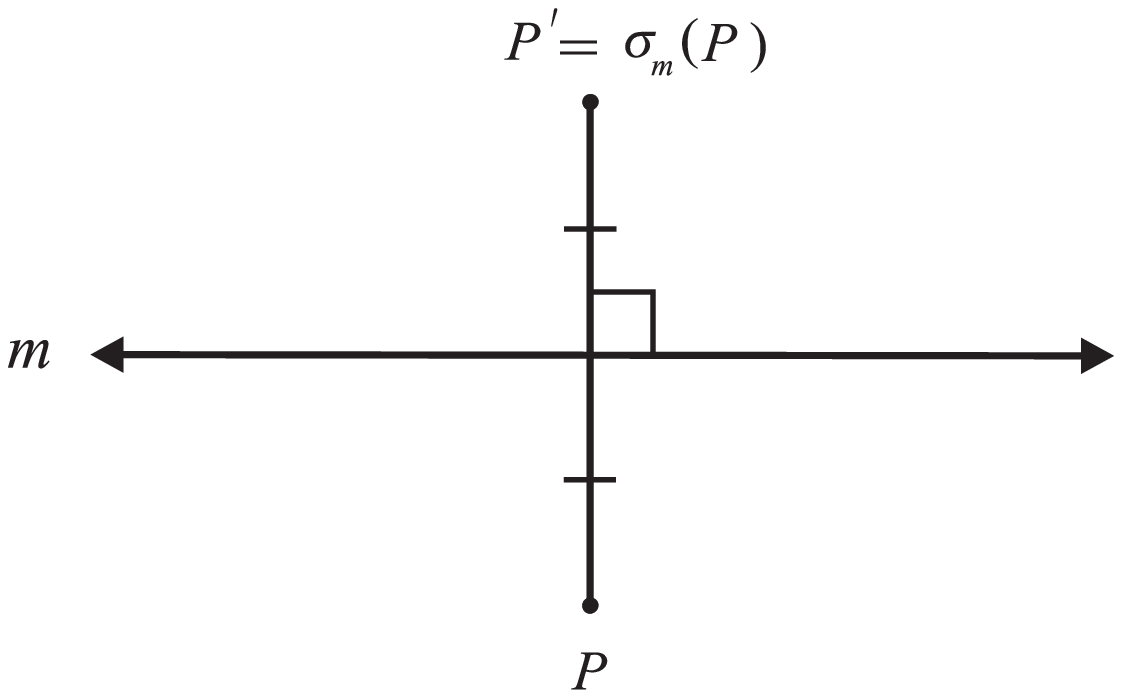}
\\\bigskip
\upshape{Figure 3. Line $m$ is the axis of
reflection.}
\end{center}
\medskip

Lines $l$ and $m$ are \textbf{parallel }if $l=m$ or $l\cap m=\varnothing $.
A \textbf{collineation} is a transformation $\phi :\mathbb{R}^{2}\rightarrow
\mathbb{R}^{2}$ that sends lines to lines; a collineation is a \textbf{%
dilatation }if it sends each line to a line parallel to it. In fact,
similarities are collineations and every dilatation that is a non-isometric
similarity is a dilation.\medskip

\section{Constructing Fixed Points}

The fixed point of a non-trivial dilation $\delta _{C,r},$ namely its center
$C,$ is easy to construct. Choose any point $A.$ If $A=\varphi _{C}\left(
A\right) ,$ then $A=C.$ So assume that $A\neq A^{\prime }=\varphi _{C}\left(
A\right) .$ If $r=1,$ then $\delta _{C,r}\left( A\right) =\varphi _{C}\ $and
$C$ is the midpoint of $\overline{AA^{\prime }}.$ If $r\neq 1,$ choose any
point $B$ off $\overleftrightarrow{AA}$. If $B=\delta _{C,r}\left( B\right) $%
, then $C=B$. So assume that $B\neq B^{\prime }=\delta _{C,r}\left( B\right)
.$ Since dilations are dilatations, $\overleftrightarrow{AB}\parallel
\overleftrightarrow{A^{\prime }B^{\prime }}$. On the other hand, $%
\overleftrightarrow{AA^{\prime }}\nparallel \overleftrightarrow{BB^{\prime }}%
;$ otherwise, $\square AA^{\prime }B^{\prime }B$ is a parallelogram and $%
A^{\prime }B^{\prime }=AB$ implies $\delta _{C,r}$ is an isometry, which is
a contradiction. Furthermore, $\delta _{C,r}(\overleftrightarrow{AA^{\prime }%
})=\overleftrightarrow{AA^{\prime }}$ since $A^{\prime }\in
\overleftrightarrow{AA^{\prime }}\,\cap \,\delta _{C,r}(\overleftrightarrow{%
AA^{\prime }})$, and similarly $\delta _{C,r}(\overleftrightarrow{BB^{\prime
}})=\overleftrightarrow{BB^{\prime }}$. Therefore $\delta _{C,r}(%
\overleftrightarrow{AA^{\prime }}\cap \overleftrightarrow{BB^{\prime }})=%
\overleftrightarrow{AA^{\prime }}\cap \overleftrightarrow{BB^{\prime }}$ and
$\overleftrightarrow{AA^{\prime }}\cap \overleftrightarrow{BB^{\prime }}=C$
(see Figure 4).

\medskip

\begin{center}
\includegraphics[
height=1.2321in,
width=2.6916in
]
{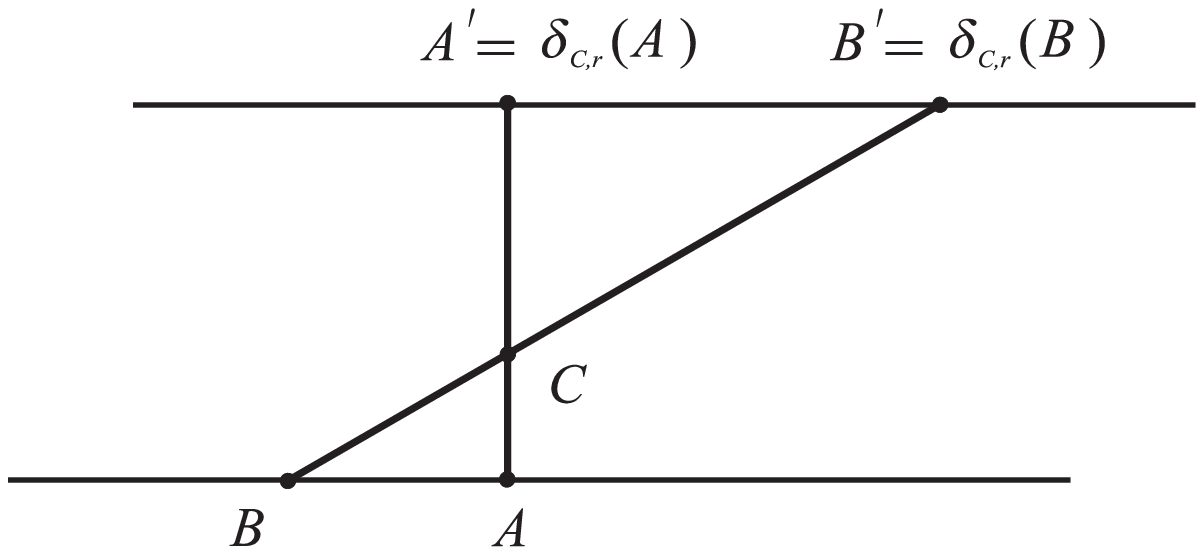}
\\\bigskip
\upshape{Figure
4. Constructing the fixed point $C$ of a non-isometric dilation.}
\end{center}
\medskip

The proof of our general fixed point construction algorithm rests heavily on
the following standard existence theorem, which we include here for
completeness:

\begin{theorem}
\label{thm4.7}Every non-isometric similarity has a fixed point.
\end{theorem}

\begin{proof}
Let $\alpha $ be a non-isometric similarity. If $\alpha $ is dilatation, it
is a dilation and has a fixed point, namely its center constructed above. So
assume that $\alpha $ is not a dilatation. Let $l$ be a line that cuts its
image $l^{\prime }=\alpha (l)$ at the point $A=l\cap l^{\prime }$ and let $%
A^{\prime }=\alpha (A).$ If $A^{\prime }=A,$ then $\alpha $ has a fixed
point and we're done. So assume $A^{\prime }\neq A.$ Then $A^{\prime }$ is
on $l^{\prime }$ and off $l.$ Let $m$ be the line through $A^{\prime }$
parallel to $l.$ Since $A^{\prime }$ is off $l,$ the parallels $l$ and $m$
are distinct. Let $m^{\prime }=\alpha (m)$. We claim that lines $l^{\prime }$
and $m^{\prime }$ are distinct parallels. If not, either $l^{\prime
}\nparallel m^{\prime }$ or $l^{\prime }=m^{\prime }$, and in either case
there exists a point $Q^{\prime }\in l^{\prime }\cap m^{\prime }=\alpha
(l)\cap \alpha (m)$ and points $Q_{1}$ on $l$ and $Q_{2}$ on $m$ such that $%
Q^{\prime }=\alpha \left( Q_{1}\right) =\alpha \left( Q_{2}\right) .$ But $l$
and $m$ are distinct parallels, hence $Q_{1}\neq Q_{2}$ and $\alpha $ is not
injective, which is a contradiction. Therefore $l^{\prime }$ and $m^{\prime
} $ are distinct parallels, as claimed. Let $B=m\cap m^{\prime }$ and let $%
B^{\prime }=\alpha (B).$ If $B^{\prime }=B,$ then $\alpha $ has a fixed
point and we're done. So assume $B^{\prime }\neq B.$ Then $B^{\prime }$ is
on $m^{\prime }$ and $B^{\prime }\neq A^{\prime }$ since $A^{\prime }$ is on
$l^{\prime }.$ Note that $\overleftrightarrow{AA^{\prime }}=l^{\prime }\neq
m^{\prime }=\overleftrightarrow{BB^{\prime }}$ and $\overleftrightarrow{%
AA^{\prime }}\Vert \overleftrightarrow{BB^{\prime }}$ by the claim. Since $%
\alpha $ is not an isometry, $\square ABB^{\prime }A^{\prime }$ is not a
parallelogram and $\overleftrightarrow{AB}\nparallel \overleftrightarrow{%
A^{\prime }B^{\prime }}$. Let $P=\overleftrightarrow{AB}\cap
\overleftrightarrow{A^{\prime }B^{\prime }}$ and let $P^{\prime }=\alpha
\left( P\right) .$ If $P^{\prime }=P,$ the proof is complete. We consider
three cases:\smallskip\

\noindent \underline{Case 1}. $A-P-B.$ Then $A^{\prime }-P-B^{\prime }$
since $\overleftrightarrow{AA^{\prime }}\Vert \overleftrightarrow{BB^{\prime
}}$ (see Figure 5) and $A^{\prime }-P^{\prime }-B^{\prime }$ since $\alpha $
preserves betweenness. Thus $A^{\prime },$ $P,$ $P^{\prime },$ and $%
B^{\prime }$ are collinear. Since $\overleftrightarrow{AA^{\prime }}\Vert
\overleftrightarrow{BB^{\prime }}$, we have $\angle PAA^{\prime }\cong
\angle PBB^{\prime }$ and $\angle PA^{\prime }A\cong \angle PB^{\prime }B$
by the Alternate Interior Angles Theorem. Hence $\triangle APA^{\prime }\sim
\triangle BPB^{\prime }$ (AA) so that $AP/BP=A^{\prime }P/B^{\prime }P$ by
the Similar Triangles Theorem. Let $r$ be the ratio of the similarity $%
\alpha ;$ then $A^{\prime }P^{\prime }=rAP$ and $B^{\prime }P^{\prime }=rBP$
so that
\begin{equation*}
\frac{A^{\prime }P^{\prime }}{B^{\prime }P^{\prime }}=\frac{rAP}{rBP}=\frac{%
AP}{BP}=\frac{A^{\prime }P}{B^{\prime }P}.
\end{equation*}%
Since $A^{\prime }P^{\prime }=A^{\prime }B^{\prime }-B^{\prime }P^{\prime }$
and $A^{\prime }P=A^{\prime }B^{\prime }-B^{\prime }P$, direct substitution
gives
\begin{equation*}
\frac{A^{\prime }B^{\prime }-B^{\prime }P^{\prime }}{B^{\prime }P^{\prime }}=%
\frac{A^{\prime }B^{\prime }-B^{\prime }P}{B^{\prime }P}.
\end{equation*}%
Then $B^{\prime }P^{\prime }=B^{\prime }P$ by algebra, and $P=P^{\prime }$
as claimed.

\begin{center}
\includegraphics[
height=1.4538in,
width=2.7216in
]
{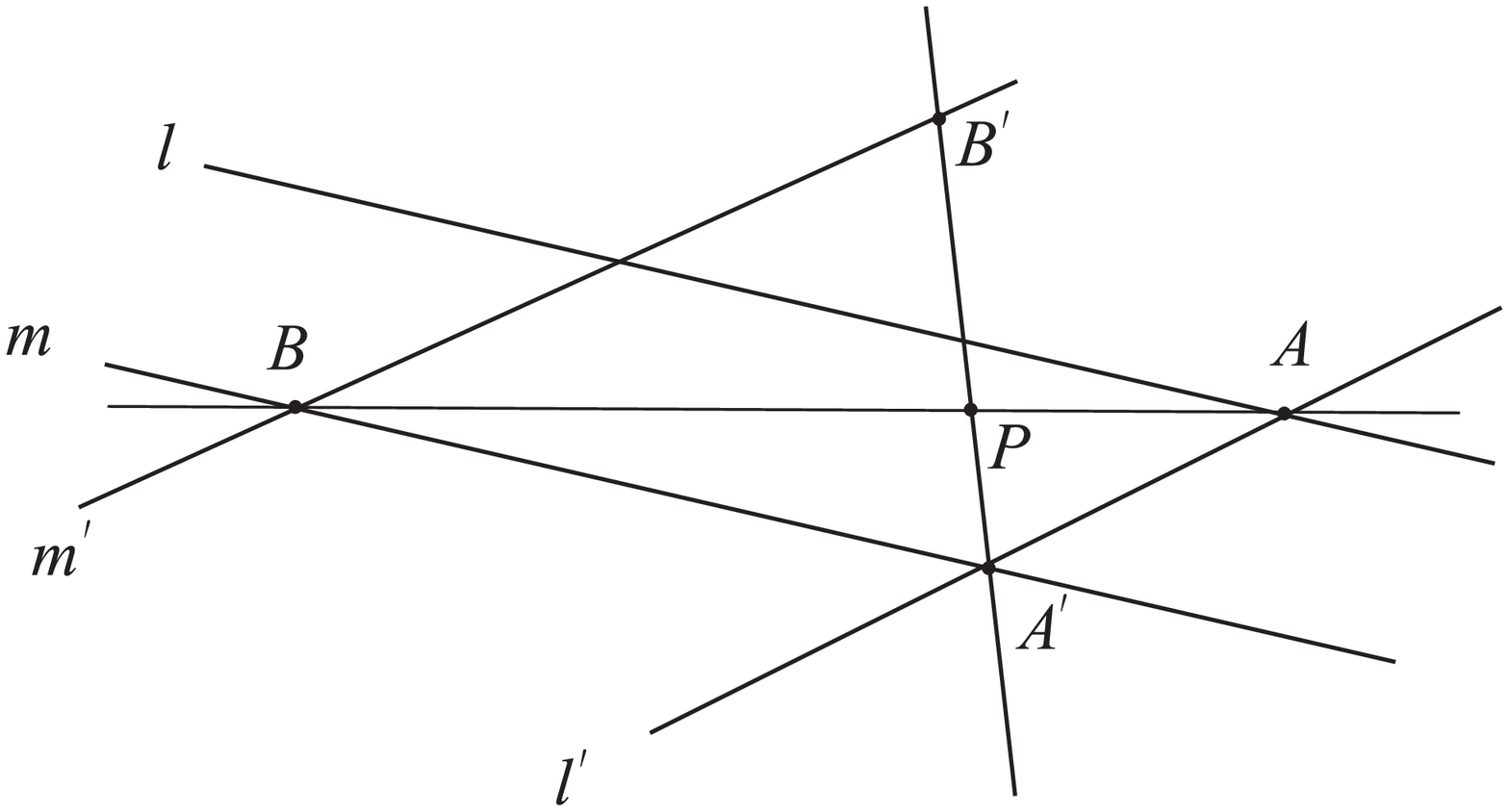}
\\\bigskip
\upshape{Figure 5.
The points $A^{\prime },$ $P,$ and $B^{\prime }$ are collinear.}
\end{center}
\medskip

\noindent \underline{Cases 2 and 3}. The proofs assuming $A-B-P$ and $P-A-B$
are similar and left to the reader. $\vspace{0.1in}$
\end{proof}

Theorem \ref{thm4.7} asserts that every non-isometric similarity $\alpha $
has a fixed point, so let's construct it. By the Classification Theorem for
Plane Similarities, $\alpha $ is a stretch, a stretch rotation, or a stretch
reflection \cite{Martin}, so the fixed point of $\alpha $ is always the
center of a stretch. The fixed point of a dilation was constructed above, so
assume that $\alpha $ is not a dilation.

We need the following technical fact:

\begin{proposition}
\label{non-sim-triangles}Let $\alpha $ be a non-isometric similarity that is
not a dilation. There exist non-degenerate triangles $\triangle PQR$ and $%
\triangle P^{\prime }Q^{\prime }R^{\prime }$ such that $P^{\prime }=\alpha
\left( P\right) ,$ $Q^{\prime }=\alpha \left( Q\right) ,$ $R^{\prime
}=\alpha \left( R\right) ,$ $\overleftrightarrow{PQ}\nparallel
\overleftrightarrow{P^{\prime }Q^{\prime }}$, and $\overleftrightarrow{QR}%
\nparallel \overleftrightarrow{Q^{\prime }R^{\prime }}.$
\end{proposition}

\begin{proof}
Since $\alpha $ is not a dilation, it is not a dilatation. Hence there
exists a line $q$ such that $q^{\prime }=\alpha (q)\nparallel q$, and it
follows that $q^{\prime \prime }=\alpha (q^{\prime })\nparallel \alpha
(q)=q^{\prime }$. Let $Q=q\cap q^{\prime }$. Choose a point $P\neq Q$ on $q$
and a point $R\neq Q$ on $q^{\prime }$. Then $P,Q$, and $R$ are distinct and
non-collinear, and so are $Q^{\prime }=\alpha (Q)=q^{\prime }\cap q^{\prime
\prime }$, $P^{\prime }=\alpha (P)\in q^{\prime }$, and $R^{\prime }=\alpha
(R)\in q^{\prime \prime }$. Hence $\overleftrightarrow{PQ}=q\nparallel
q^{\prime }=\overleftrightarrow{P^{\prime }Q^{\prime }}$ and $%
\overleftrightarrow{QR}=q^{\prime }\nparallel q^{\prime \prime }=%
\overleftrightarrow{Q^{\prime }R^{\prime }}.$\vspace{0.1in}
\end{proof}

Our fixed point construction algorithm now follows.

\begin{algorithm}
\label{construct-fixed-pt}Let $\alpha $ be a non-isometric similarity that
is not a dilation. To construct the fixed point of $\alpha $, choose
triangles $\triangle PQR$ and $\triangle P^{\prime }Q^{\prime }R^{\prime }$
that satisfy the conditions of Proposition \ref{non-sim-triangles} and
proceed as follows:

\begin{enumerate}
\item Construct line $n$ through $R$ parallel to $m=\overleftrightarrow{PQ}$.

\item Construct line $n^{\prime}$ through $R^{\prime}$ parallel to $%
m^{\prime }=\overleftrightarrow{P^{\prime}Q^{\prime}}$.

\item Let $D=m\cap m^{\prime}$ and $E=n\cap n^{\prime}$.

\item Construct line $a=\overleftrightarrow{DE}$ \textup{(see Figure 6)}.

\item Interchange $P$ and $R$ and repeat steps 1--4 to construct line $b$.

\item Then $C=a\cap b$ is the fixed point of $\alpha $.
\end{enumerate}
\end{algorithm}

\medskip

\begin{center}
\includegraphics[
height=3.1586in,
width=4.6469in
]
{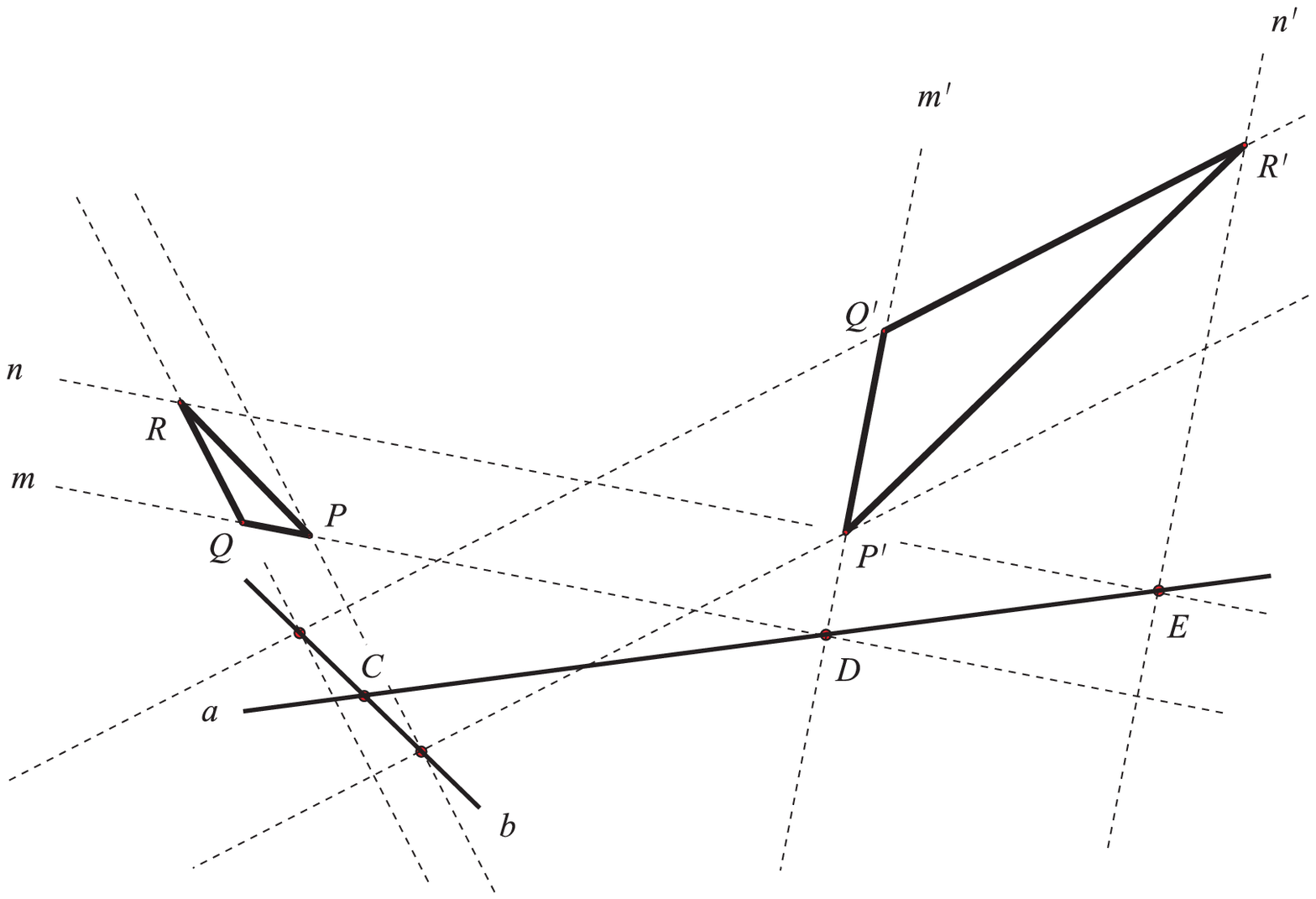}
\\\bigskip
\upshape{Figure 6. Constructing the
fixed point $C$ of a stretch rotation.}
\end{center}
\medskip

The proof of Algorithm \ref{construct-fixed-pt} is a consequence of the
following lemma:

\begin{lemma}
\label{stretch-lemma}Let $\alpha$ be a non-isometric similarity and let $C$
be the fixed point of $\alpha$ given by Theorem \ref{thm4.7}. Let $l$ and $m$
be distinct parallel lines off $C$, let $l^{\prime}=\alpha\left( l\right) $,
and let $m^{\prime}=\alpha\left( m\right) .$ If $l\cap l^{\prime}=D$ and $%
m\cap m^{\prime}=E,$ then $C$, $D$, and $E$ are collinear.
\end{lemma}

\begin{proof}
Let $\alpha $ be a non-isometric similarity and let $C$ be the fixed point
of $\alpha $ given by Theorem \ref{thm4.7}. Then $\alpha $ is not a dilation
since $l^{\prime }=\alpha (l)\nparallel l$ by assumption. Hence by the
Classification Theorem for Similarities, $\alpha =\rho _{C,\Theta }\circ \xi
_{C,r}$ or $\alpha =\sigma _{n}\circ \xi _{C,r}$, where in the latter case, $%
C$ is on $n$. Since $l$ and $m$ are distinct parallel lines off $C$, there
exists a dilation $\delta _{C,s}$ such that $\delta _{C,s}\left( l\right) =m$%
. By definition, either $\delta _{C,s}=\xi _{C,s}$ or $\delta _{C,s}=\varphi
_{C}\circ \xi _{C,s}$. The reader can easily check that $\xi _{C,s}$
commutes with $\rho _{C,\Theta },$ $\sigma _{n},$ and $\xi _{C,r}$. Hence $%
\xi _{C,s}$ commutes with $\alpha $. Furthermore, $\varphi _{C}$ commutes
with $\rho _{C,\Theta }$, $\sigma _{n}$, and $\xi _{C,r}$. Hence $\varphi
_{C}$ commutes with $\alpha $. Consequently, $\delta _{C,s}$ commutes with $%
\alpha $ and $\delta _{C,s}\left( l^{\prime }\right) =\delta _{C,s}\left(
\alpha \left( l\right) \right) =\alpha \left( \delta _{C,s}\left( l\right)
\right) =\alpha \left( m\right) =m^{\prime }.$ Therefore $\delta
_{C,s}\left( D\right) =\delta _{C,s}\left( l\cap l^{\prime }\right) =m\cap
m^{\prime }=E,$ and it follows that $C$, $D$, and $E$ are collinear.\vspace{%
0.1in}
\end{proof}

\begin{theorem}
Let $\alpha$ be a non-isometric similarity and let $C$ be the point given by
Algorithm \ref{construct-fixed-pt}. Then $C$ is the fixed point of $\alpha$.
\end{theorem}

\begin{proof}
Referring to the notation in Algorithm 1, note that $\alpha \left( m\right)
=m^{\prime }$ since $\alpha \left( P\right) =P^{\prime }$ and $\alpha \left(
Q\right) =Q^{\prime }.$ Thus $\alpha \left( n\right) =n^{\prime }$ since $%
\alpha \left( R\right) =R^{\prime }$ and $\alpha $ maps parallel lines to
parallel lines. By Lemma \ref{stretch-lemma}, $C,$ $D,$ and $E$ are
collinear. A similar argument shows that $C,$ $F,$ and $G$ are also
collinear. Therefore $C=\overleftrightarrow{DE}\cap \overleftrightarrow{FG}.$%
\vspace{0.1in}
\end{proof}

\begin{example}
\textup{The two triangles in Figure 7 below are related by a stretch
reflection} $\alpha $\textup{\ whose fixed point $C$ was given by Algorithm %
\ref{construct-fixed-pt}. To construct the axis of reflection, let }$P$%
\textup{\ be a vertex and let }$P^{\prime }=\alpha \left( P\right) .$\textup{%
\ Then the axis of reflection is the bisector of} $\angle PCP^{\prime }.$%
\bigskip
\end{example}

\medskip

\begin{center}
\includegraphics[
height=2.8167in,
width=4.8957in
]
{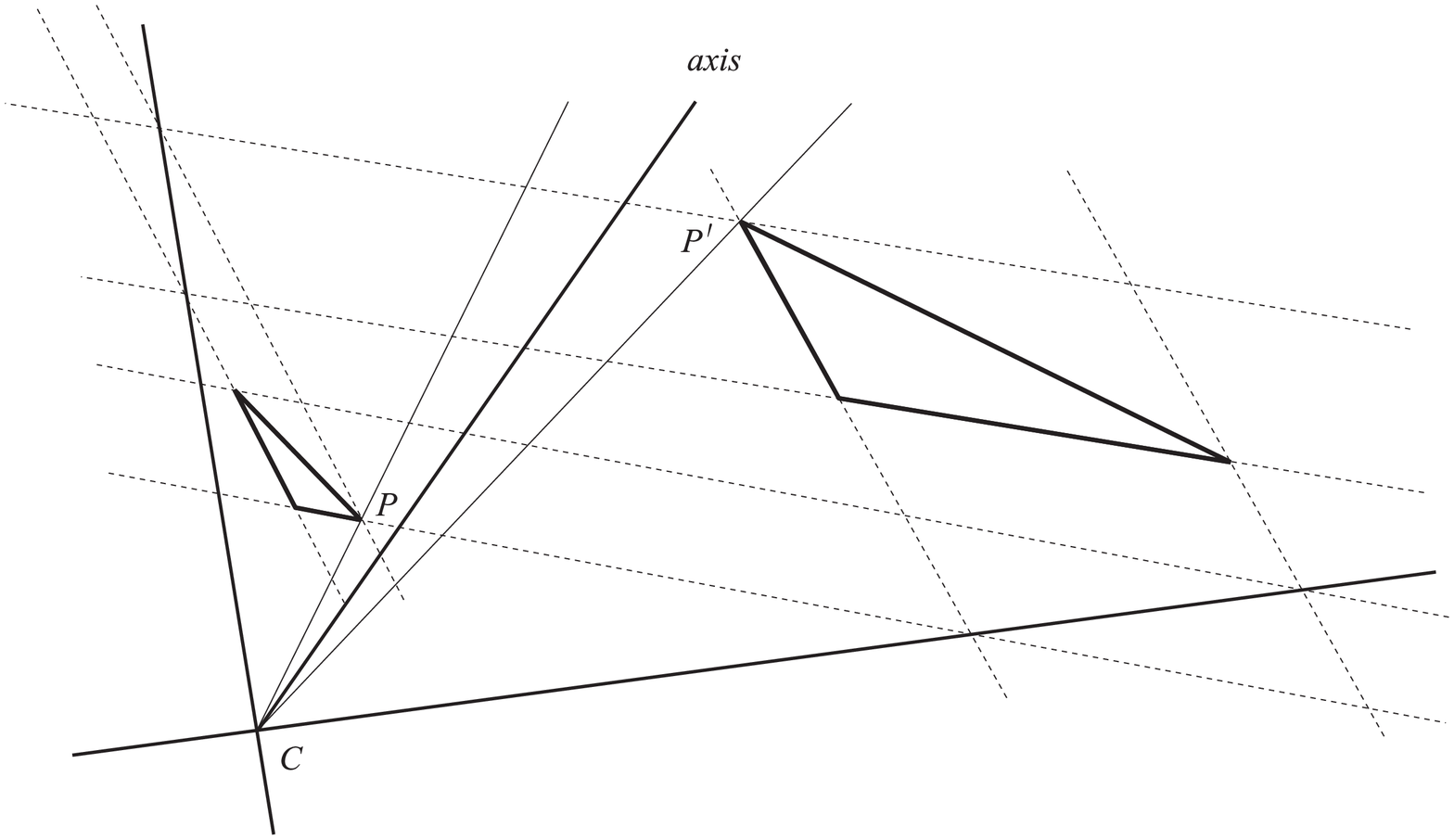}
\\\bigskip
\upshape{Figure 7. Constructing the
fixed point $C$ and axis of a stretch reflection.}
\end{center}
\medskip

\section{Conclusion}

Given a non-isometric plane similarity $\alpha$, we have presented an
algorithm for constructing the fixed point of $\alpha.$ Our algorithm can be
easily applied using a reflecting instrument such as MIRA. Since such
instruments are readily available in a typical high school classroom, our
algorithm gives students of plane geometry direct hands-on access to the
fixed point of a non-isometric similarity. Hands-on activities such as this
create the potential for a deeper and more complete understanding of plane
similarities and their properties.

\end{document}